\documentclass[11pt]{article}

\usepackage[margin=1in]{geometry}
\usepackage{amsmath,amssymb,amsthm}
\usepackage{mathtools,mathrsfs}
\usepackage{enumitem}
\usepackage{hyperref}
\usepackage{cleveref}
\usepackage{microtype}
\usepackage{authblk}

\hypersetup{
  colorlinks=true,
  linkcolor=blue,
  citecolor=blue,
  urlcolor=blue,
  pdftitle={Exact ReLU realization of binary affine refinement iterates via reflection folding and cone switching},
  pdfauthor={Boldsaikhan Bolorkhuu and Tsogtgerel Gantumur},
  pdfkeywords={ReLU networks, binary refinement, affine refinement, reflection folding, residual memory, tent map, continuous piecewise linear functions}
}

\newtheorem{theorem}{Theorem}[section]
\newtheorem{proposition}[theorem]{Proposition}
\newtheorem{lemma}[theorem]{Lemma}
\newtheorem{corollary}[theorem]{Corollary}
\theoremstyle{definition}

\newtheorem{remark}[theorem]{Remark}
\newtheorem{example}[theorem]{Example}

\newcommand{\R}{\mathbb{R}}
\newcommand{\N}{\mathbb{N}}
\newcommand{\Z}{\mathbb{Z}}
\newcommand{\T}{\mathbb{T}}
\newcommand{\trans}{\mathrm{T}}
\newcommand{\emb}{\iota_\Gamma}
\newcommand{\mem}{\mathcal M}
\newcommand{\supp}{\operatorname{supp}}
\newcommand{\vect}{\operatorname{Vec}}
\newcommand{\ReLU}{\operatorname{ReLU}}
\newcommand{\Ups}{\Upsilon}
\newcommand{\Lip}{\operatorname{Lip}}
\newcommand{\clip}{\operatorname{clip}}
\newcommand{\e}{\mathbf{e}}
\newcommand{\doi}[1]{\href{https://doi.org/#1}{DOI: #1}}

\title{Exact ReLU realization of binary affine refinement iterates via reflection folding and cone switching}
\author[1,2,3]{Tsogtgerel Gantumur}
\author[2]{Boldsaikhan Bolorkhuu}
\affil[1]{McGill University}
\affil[2]{National University of Mongolia}
\affil[3]{Institute of Mathematics and Digital Technology, Mongolian Academy of Sciences}
\date{July 25, 2026}

\begin{document}
\maketitle

\begin{abstract}
We study vector-valued binary affine refinement operators with finitely
supported matrix masks and compactly supported continuous piecewise linear
input and forcing data.  We prove that every finite refinement iterate admits
an exact ReLU realization of fixed width and depth linear in the number of
iterations.  No separation of the forcing profile from the binary cell seams
is required.

The main mechanism is universal reflection doubling.  Pairing each residual
profile with its reflection replaces the two binary transition matrices by one
fixed block matrix together with a fixed swap involution.  The cell-seam
identity makes the two branch candidates agree at the tent fold, while their
swap-odd component is bounded linearly by the distance to the fold.  This
permits exact branch selection by a fixed continuous piecewise linear cone
switch, without multiplication by a variable selector.

The resulting primal recursion requires the residual orbit in reverse order.
We obtain exact backward replay from the residual memory controller developed
previously for affine refinement, interpreted here through the reflection
quotient of circle doubling.  The construction propagates the full vectorized
profiles rather than decomposing the input and forcing into reference atoms.
We also treat stage-dependent forcing from a fixed finite-dimensional family
and show that genuine reflection equivariance reduces the doubled cascade to a
single parity sector.
\end{abstract}



\section{Introduction}
\label{sec:introduction}

\subsection{Background and motivation}

Neural-network approximation theory seeks structural explanations for the expressive power of deep
networks and, in particular, for the role of depth in representing highly recursive or highly
oscillatory functions \cite{approx,nnapprox}.  Refinement operators provide a particularly clean
class of examples.  In the scalar binary setting, Daubechies et al. proved that finite refinement
iterates of compactly supported continuous piecewise linear functions admit exact ReLU realizations
of fixed width and depth growing linearly with the number of refinement steps \cite{source}.
Subsequent work developed exact loop controllers for homogeneous vector-valued \(M\)-ary
refinement and a residual memory controller for affine forcing \cite{loop,affine-memory}.

The present paper addresses a peculiarity of the binary affine problem.  In the
\(M\)-ary construction of \cite{affine-memory}, when \(M\ge3\) one can choose a
nonzero offset compatible with the residual dynamics, and ordinary and shifted
frames provide seam-safe forcing readouts.  For \(M=2\), no such nonzero offset
exists, so that argument requires the forcing profile to be separated from the
ordinary cell seams.  We remove this restriction using the reflection geometry
of the two binary branches.

The starting observation is the familiar tent-map folding of binary residual dynamics.  Let
\[
  Q(x)=\lfloor 2x\rfloor,
  \qquad
  R(x)=2x-Q(x),
  \qquad x\in[0,1),
\]
and define the tent map
\begin{equation}
\label{eq:intro-tent}
  \tau(x)=
  \begin{cases}
    2x, & 0\le x\le \frac12,\\
    2-2x, & \frac12\le x\le1.
  \end{cases}
\end{equation}
Then \(\tau^m(x)\in\{R^m(x),1-R^m(x)\}\) for \(x\in[0,1)\) and \(m\ge0\).  Consequently, if a scalar
profile \(h\) is reflection-even, in the sense that
\(h(1-t)=h(t)\), then \(h(R^m(x))=h(\tau^m(x))\).
This elementary identity is closely related to the sawtooth constructions used in
depth-separation arguments~\cite{telgarsky}.

The complementary reflection-odd case already contains the local switching mechanism used
below.  Suppose that \(h(1-t)=-h(t)\) and that \(h\) satisfies the natural endpoint compatibility
condition \(h(0)=h(1)=0\).  On the left branch of the tent map one has
\(h(R(x))=h(\tau(x))\), whereas on the right branch one has
\(h(R(x))=-h(\tau(x))\).  The graph of \(x\mapsto h(\tau(x))\) lies in a double cone
with vertex at \((\frac12,0)\).  One can therefore define a CPwL map
\(\mathcal C(x,y)\) that returns \(y\) on the left cone and \(-y\) on the right cone,
with a continuous piecewise linear transition elsewhere.  In particular,
\[
  h(R(x))=\mathcal C\bigl(x,h(\tau(x))\bigr).
\]
We refer to this elementary device as a cone switch.

For a matrix cascade, symmetry of a terminal scalar factor does not remove the
branch dependence of the matrix product.  The decisive step is therefore to
double the entire residual fiber.  For a vectorized residual profile \(G\), set
\[
  G^{\#}(x)
  =
  \begin{bmatrix}
    G(x)\\
    G(1-x)
  \end{bmatrix}.
\]
The two binary branches are then represented by one fixed block matrix, with
the right branch followed by a fixed swap involution.  At the tent fold, the
two candidates agree by the cell-seam identity satisfied by the vectorization
of a continuous global function.  Lipschitz continuity consequently bounds
the swap-odd component by a constant multiple of
\(\lvert x-\frac12\rvert\).  After decomposing the doubled fiber into the even
and odd subspaces of the swap, the scalar cone flip applies coordinatewise to
the odd component.  The discontinuous binary matrix selector is thereby
replaced by one fixed block matrix and one fixed CPwL switching module.

The resulting folded recursion is primal.  Although the functional iteration
proceeds forward in the refinement level, its pointwise evaluation is
inside-out: to evaluate the \(n\)-th iterate at \(x\), one first evaluates the
initial profile at \(\tau^n(x)\) and then applies the affine updates along
\[
  \tau^{n-1}(x),\ldots,\tau(x),x.
\]
Storing the entire residual orbit would make the width grow with \(n\), while
recomputing it at every stage would destroy linear depth.  We therefore use the
residual memory controller of \cite{affine-memory}, with circle doubling as a
two-sheeted lift of the tent dynamics.  An injective skew product on a
polygonal loop stores the missing branch information, and its CPwL inverse
replays the folded residual orbit exactly in reverse order.

The two mechanisms have complementary roles.  Reflection doubling and cone
switching remove the discontinuous branch selector, while residual memory
supplies the reverse chronology required by the primal recursion.  The combined
construction imposes no seam-separation condition on the forcing profile.

Memory is not intrinsically necessary in the homogeneous problem.  Besides
the earlier forward constructions, the cone switch also admits an adjoint
homogeneous realization.  For endpoint-zero reference atoms, the swap-odd
adjoint candidate satisfies the same cone estimate, so the adjoint cascade
can be propagated forward without residual memory.

We nevertheless use the primal folded-memory formulation throughout.  It
gives a unified treatment of the homogeneous and affine problems, displays
the reflection-folding mechanism directly, and propagates the complete
vectorized profiles without an atomic decomposition.

\subsection{Main results}

For positive integers \(w,D,d,N\), let
\(\Ups_{w,D}(\ReLU;d,N)\) denote the class of outputs of fully connected
ReLU networks with width \(w\), depth \(D\), input dimension \(d\), and
output dimension \(N\).

We study the binary affine refinement operator
\begin{equation}
\label{eq:intro-W}
  (W\gamma)(t)
  =
  \sum_{j\in\Z}A_j\gamma(2t-j)+B(t),
\end{equation}
where \(A_j\in\R^{p\times p}\) and only finitely many \(A_j\) are nonzero.
We assume that the homogeneous part preserves a fixed support window
\([0,L]\), with \(L\in\N\).

\begin{theorem}[Main theorem, informal form]
\label{thm:intro-main}
Suppose that the homogeneous part of \(W\) preserves \([0,L]\), and that
\(\gamma,B:\R\to\R^p\) are CPwL and supported in this window.  Then there
exist constants \(C_0,C_1>0\), independent of \(n\), such that
\[
  W^n\gamma
  \in
  \Ups_{C_0,C_1n}(\ReLU;1,p),
  \qquad n\ge1.
\]
The realizing networks may be chosen with weights and biases bounded by
\(C_2\Lambda^n\), where \(C_2,\Lambda>0\) depend only on the mask, the
support window, and the fixed CPwL data.
\end{theorem}

The theorem is exact in real arithmetic; no claim is made about
finite-precision stability or bit complexity.  Unlike the atomic
decompositions used in earlier refinement constructions, the proof propagates
the complete finite-dimensional vectorization of the input and forcing
profiles on \([0,1]\).  The global function is then recovered from its unit-cell
components by the standard clamped-gluing argument.  The same architecture
also accommodates stage-dependent forcing from a fixed finite-dimensional
family, while a compatible reflection symmetry reduces the universal doubled
cascade to a single parity sector.

\subsection{Contribution and relation to earlier constructions}

The homogeneous scalar binary realization theorem is due to \cite{source},
while homogeneous vector-valued refinement was treated by the loop-controller
construction of \cite{loop}.  The residual memory controller used here is the
binary specialization of the affine construction in
\cite{affine-memory}.  The new contribution consists of the following three
mechanisms.

\begin{enumerate}[label=(\roman*),leftmargin=2.5em]
  \item \emph{Universal reflection doubling.}
  Adjoining the reflected residual profile places every binary matrix
  cascade in a doubled fiber carrying a fixed swap involution.

  \item \emph{Folded matrix recursion.}
  In the doubled fiber, the two binary transition matrices are replaced by
  one fixed block matrix, followed on the right branch of the tent map by
  the swap involution.

  \item \emph{Cone switching.}
  The swap-odd component is bounded by a constant multiple of the distance
  to the tent fold.  A fixed CPwL switch can therefore choose exactly
  between the identity and the swap on the resulting quantitative cone,
  without multiplying the evolving state by a variable selector.
\end{enumerate}

The folded recursion is primal and must be evaluated in reverse residual
order.  We handle this chronology using the residual memory controller of
\cite{affine-memory}.  Its polygonal loop provides a two-sheeted lift of the
tent dynamics, while a global affine readout recovers the tent coordinate.
Thus the earlier memory mechanism supplies backward replay, whereas
reflection doubling and cone switching provide the new treatment of the
binary branch ambiguity.

\subsection{Organization}

Section~\ref{sec:preliminaries} introduces binary vectorization, transition
matrices, seam compatibility, and universal reflection doubling.
Section~\ref{sec:cone-switch} constructs the cone switch and establishes the
quantitative estimate needed for exact switching.
Section~\ref{sec:memory} recalls the binary residual memory controller and
identifies its folded tent-map readout.
Section~\ref{sec:unit-realization} combines these ingredients to realize the
affine cascade on one unit interval.
Section~\ref{sec:global} proves the global theorem and its stage-dependent
extension.
Section~\ref{sec:reflection-equivariant} treats genuine reflection symmetry
and the resulting reduction of the cascade fiber.
Appendix~\ref{app:adjoint-homogeneous} records a memory-free adjoint
realization for homogeneous atoms, while Appendix~\ref{app:fan-switch}
shows that the binary cone switch is the two-branch case of a fixed-arity
fan-switch construction based on polyphase shifts.

\section{Binary vectorization and reflection folding}
\label{sec:preliminaries}

\subsection{Network and CPwL conventions}

A map between finite-dimensional Euclidean spaces is called CPwL if it is continuous and affine on
each cell of a locally finite polyhedral subdivision.  Every fixed CPwL module used below admits an
exact finite ReLU realization.  For one-dimensional scalar maps this follows from the hinge
representation, while fixed-dimensional polyhedral modules may be represented by the standard
continuous piecewise affine finite-element construction \cite{relu-fem}.  
When a map defined on \([0,1]\) is said to have a ReLU realization, we
mean that it is the restriction to \([0,1]\) of a globally defined network output.
We use repeatedly that
affine precomposition, finite parallelization, finite linear combinations, and composition with a fixed
CPwL module preserve fixed width up to constant factors and preserve depth \(O(n)\).

All norms on fixed finite-dimensional spaces are equivalent.  We use the max norm for explicit
estimates, together with the induced matrix norm.  Constants are not optimized.

\subsection{Support window and vectorization}

Let
\begin{equation}
\label{eq:V-definition}
(Vf)(t)=\sum_{j\in\Z}A_jf(2t-j),
\qquad A_j\in\R^{p\times p},
\end{equation}
and suppose only finitely many \(A_j\) are nonzero.  Fix an integer \(L\ge1\) and assume the support
window condition
\begin{equation}
\label{eq:support-window}
\supp f\subset[0,L]
\quad\Longrightarrow\quad
\supp Vf\subset[0,L].
\end{equation}
The compactly supported input and forcing data are assumed to lie in this window.  Continuity on
\(\R\) then implies that they vanish at \(0\) and \(L\).

Set
\(E=(\R^p)^L\),
and for a continuous \(f:\R\to\R^p\) supported in \([0,L]\), define its vectorization by
\begin{equation}
\label{eq:vectorization}
(\vect f)(x)
=
\begin{bmatrix}
f(x)\\f(x+1)\\ \vdots\\f(x+L-1)
\end{bmatrix}
\in E,
\qquad x\in[0,1].
\end{equation}
We index the blocks of \(E\) by \(k=0,\ldots,L-1\).

For \(q\in\{0,1\}\), define the block transition matrix \(T_q:E\to E\) by
\begin{equation}
\label{eq:Tq-definition}
(T_qu)_k
=
\sum_{\ell=0}^{L-1}A_{q+2k-\ell}u_\ell,
\qquad k=0,\ldots,L-1.
\end{equation}
The finite support of the mask makes each sum finite, and the support-window hypothesis ensures
that this finite vectorization contains all required coordinates.

Throughout, the digit representation \(2x=Q(x)+R(x)\) is used only for
\(x\in[0,1)\); the branch formulas on the closed half intervals are
understood by continuous extension.

\begin{lemma}[Binary cascade identity]
\label{lem:binary-cascade}
Let \(G=\vect f\).  For \(x\in[0,1]\), one has
\begin{equation}
\label{eq:binary-cascade}
\vect(Vf)(x)
=
\begin{cases}
T_0G(2x),&0\le x\le\frac12,\\
T_1G(2x-1),&\frac12\le x\le1.
\end{cases}
\end{equation}
The two expressions agree at \(x=\frac12\).
\end{lemma}

\begin{proof}
For \(x\in[0,1)\), write \(2x=Q(x)+R(x)\), where \(Q(x)\in\{0,1\}\) and
\(R(x)\in[0,1)\).  For the \(k\)-th component, we have
\[
(Vf)(x+k)
=
\sum_jA_jf(R(x)+Q(x)+2k-j).
\]
Setting \(\ell=Q(x)+2k-j\) gives
\[
(Vf)(x+k)
=
\sum_{\ell=0}^{L-1}A_{Q(x)+2k-\ell}G_\ell(R(x)),
\]
which is \((T_{Q(x)}G(R(x)))_k\).  This yields the two branch formulas.  Since \(Vf\) is continuous,
their one-sided values at \(x=\frac12\) agree.
\end{proof}

The endpoint agreement in the preceding proof will be used repeatedly.

\begin{lemma}[Cell-seam identity]
\label{lem:cell-seam}
If \(G=\vect f\) for a continuous compactly supported function \(f\), then it holds that
\begin{equation}
\label{eq:cell-seam}
T_0G(1)=T_1G(0).
\end{equation}
\end{lemma}

\begin{proof}
Both sides are the value of \(\vect(Vf)\) at \(x=\frac12\), obtained from the left and right branch
formulas in \eqref{eq:binary-cascade}.
\end{proof}

\subsection{Affine vectorized recursion}

Let
\[
(Wf)(t)=Vf(t)+B(t),
\]
where \(B:\R\to\R^p\) is compactly supported and CPwL, and set
\(b=\vect B\).
If \(f_m=W^mf_0\) and \(G_m=\vect f_m\), then we have
\begin{equation}
\label{eq:affine-vector-recursion}
G_{m+1}(x)
=
\begin{cases}
T_0G_m(2x)+b(x),&0\le x\le\frac12,\\
T_1G_m(2x-1)+b(x),&\frac12\le x\le1.
\end{cases}
\end{equation}
Every \(G_m\) is CPwL and satisfies the seam identity
\begin{equation}
\label{eq:Gm-seam}
T_0G_m(1)=T_1G_m(0).
\end{equation}

\subsection{Reflected doubling for endpoint-compatible scalar profiles}

The reflection-even and reflection-odd cases can be treated simultaneously.
For a scalar CPwL profile \(h:[0,1]\to\R\) satisfying \(h(0)=h(1)\), 
define
\[
  h^\#(x)
  :=
  \begin{bmatrix}
    h(x)\\
    h(1-x)
  \end{bmatrix},
\]
and let \(J_2(a,b)=(b,a)\).  Thus
\(h^\#(1-x)=J_2h^\#(x)\).
For this scalar prototype, set \(R(1)=1\), and for \(m\ge0\), let
\[
  H_m(x)
  :=
  \begin{bmatrix}
    h(R^m(x))\\
    h(R^m(1-x))
  \end{bmatrix}.
\]
Then each \(H_m\) is CPwL, \(H_0=h^\#\), and
\begin{equation}
\label{eq:scalar-folded-recursion}
  H_{m+1}(x)
  =
  \begin{cases}
    H_m(\tau(x)),&0\le x\le\frac12,\\
    J_2H_m(\tau(x)),&\frac12\le x\le1.
  \end{cases}
\end{equation}
The two candidates in \eqref{eq:scalar-folded-recursion} agree at the
turning point.
Since \(H_m\) is CPwL, its
swap-odd component lies in a double cone at the fold, and the cone switch realizes the two
branches by one fixed CPwL update.  Thus a general endpoint-compatible profile is handled by
a single doubled state; separate even and odd controllers are unnecessary.  The matrix
construction below is the corresponding fiber-valued version of this recursion.

This scalar discussion serves only as a prototype for the folding and cone-switching
mechanism.  The realization theorem below does not assume a scalar terminal profile or
the condition \(h(0)=h(1)\).  In the vectorized matrix cascade, the corresponding
compatibility is instead the cell-seam identity \eqref{eq:cell-seam}, inherited from the
continuity of the underlying function on the line.

\subsection{Universal reflection doubling}

Let \(D=pL\), and set
\(\widehat E=E\oplus E\cong\R^{2D}\).
Define the swap involution \(\mathcal J:\widehat E\to\widehat E\) by
\begin{equation}
\label{eq:swap}
\mathcal J
\begin{bmatrix}u\\v\end{bmatrix}
=
\begin{bmatrix}v\\u\end{bmatrix},
\qquad \mathcal J^2=I.
\end{equation}
For a residual profile \(G:[0,1]\to E\), define
\begin{equation}
\label{eq:G-sharp}
G^{\#}(x)
=
\begin{bmatrix}
G(x)\\G(1-x)
\end{bmatrix}.
\end{equation}
Then \(G^{\#}(1-x)=\mathcal JG^{\#}(x)\).
Introduce the fixed block matrix
\begin{equation}
\label{eq:T-hat}
\widehat T
=
\begin{bmatrix}
T_0&0\\0&T_1
\end{bmatrix},
\end{equation}
and the doubled forcing
\begin{equation}
\label{eq:b-sharp}
b^{\#}(x)
=
\begin{bmatrix}
b(x)\\b(1-x)
\end{bmatrix}.
\end{equation}

\begin{proposition}[Folded doubled recursion]
\label{prop:folded-recursion}
Let \(G_m^{\#}\) be the reflected double of the vectorized affine iterate.  Then we have
\begin{equation}
\label{eq:folded-recursion}
G_{m+1}^{\#}(x)
=
\begin{cases}
\widehat T G_m^{\#}(\tau(x))+b^{\#}(x),&0\le x\le\frac12,\\[1mm]
\mathcal J\widehat T G_m^{\#}(\tau(x))+b^{\#}(x),&\frac12\le x\le1.
\end{cases}
\end{equation}
At \(x=\frac12\), the two matrix candidates agree:
\begin{equation}
\label{eq:fold-agreement}
\widehat T G_m^{\#}(1)
=
\mathcal J\widehat T G_m^{\#}(1).
\end{equation}
\end{proposition}

\begin{proof}
Suppose first that \(0\le x\le\frac12\), so \(\tau(x)=2x\).  The first component of the doubled
recursion is
\[
T_0G_m(\tau(x))+b(x).
\]
Since \(1-x\ge\frac12\), the second component is
\[
T_1G_m(2(1-x)-1)+b(1-x)
=T_1G_m(1-\tau(x))+b(1-x).
\]
This is the first branch in \eqref{eq:folded-recursion}.

If \(\frac12\le x\le1\), then \(\tau(x)=2-2x\).  The first component is
\[
T_1G_m(2x-1)+b(x)
=T_1G_m(1-\tau(x))+b(x),
\]
and the second component is
\[
T_0G_m(2-2x)+b(1-x)
=T_0G_m(\tau(x))+b(1-x).
\]
This is obtained by applying \(\mathcal J\) to
\(\widehat T G_m^{\#}(\tau(x))\), proving the right branch formula.

Finally, we have
\[
\widehat T G_m^{\#}(1)
=
\begin{bmatrix}
T_0G_m(1)\\T_1G_m(0)
\end{bmatrix}.
\]
The two components agree by \eqref{eq:Gm-seam}, so the vector is fixed by \(\mathcal J\).
\end{proof}

\section{Exact cone switching at the fold}
\label{sec:cone-switch}

The recursion \eqref{eq:folded-recursion} replaces the discontinuous choice between \(T_0\) and
\(T_1\) by the choice between the identity and the fixed involution \(\mathcal J\).  The candidates
agree at the fold.  This section constructs a global CPwL map that performs the switch exactly on
the cone of states generated by the cascade.

\subsection{A scalar cone flip}

Let \(\sigma(s)=\max\{s,0\}\).  For \(K>0\), define
\begin{equation}
\label{eq:scalar-cone-flip}
\phi_K(s,a)
=
-a+\sigma(a-Ks)-\sigma(-a-Ks),
\qquad (s,a)\in\R^2.
\end{equation}

\begin{lemma}[Scalar cone flip]
\label{lem:scalar-cone-flip}
The map \(\phi_K\) is CPwL and satisfies
\begin{align}
\phi_K(s,a)&=a,
&&s\le0,\quad |a|\le K|s|,
\label{eq:cone-left}\\
\phi_K(s,a)&=-a,
&&s\ge0,\quad |a|\le K|s|.
\label{eq:cone-right}
\end{align}
Moreover, \(\phi_K(s,-a)=-\phi_K(s,a)\).
\end{lemma}

\begin{proof}
If \(s\le0\) and \(|a|\le-Ks\), then both \(a-Ks\) and \(-a-Ks\) are nonnegative.  Hence
\[
\phi_K(s,a)
=-a+(a-Ks)-(-a-Ks)=a.
\]
If \(s\ge0\) and \(|a|\le Ks\), then both \(a-Ks\) and \(-a-Ks\) are nonpositive, so
\(\phi_K(s,a)=-a\).  The oddness in \(a\) follows directly from the formula.
\end{proof}

\subsection{The involution cone switch}

Let \(Y\) be a finite-dimensional real vector space and let
\(\mathscr J:Y\to Y\) be a linear involution.  Since
\(\mathscr J^2=I\), there exists a linear isomorphism
\[
  S:Y\longrightarrow\R^{d_+}\oplus\R^{d_-}
\]
such that
\begin{equation}
\label{eq:involution-diagonalization}
  S\mathscr J S^{-1}(u_+,u_-)
  =
  (u_+,-u_-).
\end{equation}
For any finite coordinate vector \(u=(u_i)\), write
\(\Phi_K(s,u):=(\phi_K(s,u_i))_i\).

\begin{proposition}[Involution cone switch]
\label{prop:involution-cone-switch}
For \(K>0\), define
\begin{equation}
\label{eq:general-involution-switch}
  \mathcal C_K^{\mathscr J}(x,u)
  :=
  S^{-1}
  \begin{bmatrix}
    u_+\\
    \Phi_K(x-\frac12,u_-)
  \end{bmatrix}
  \qquad\text{with}\quad
  (u_+,u_-)=Su.
\end{equation}
Then \(\mathcal C_K^{\mathscr J}\) is globally CPwL and satisfies
\begin{align}
  \mathcal C_K^{\mathscr J}(x,u)
  &=
  u,
  &&
  x\le\textstyle\frac12,\quad
  \|u_-\|_\infty\le K(\frac12-x),
  \label{eq:general-switch-left}
  \\
  \mathcal C_K^{\mathscr J}(x,u)
  &=
  \mathscr Ju,
  &&
  x\ge\textstyle\frac12,\quad
  \|u_-\|_\infty\le K(x-\frac12).
  \label{eq:general-switch-right}
\end{align}
It has an exact ReLU realization of fixed depth and width depending only
on \(\dim Y\), with weights bounded by \(C_S(1+K)\), where \(C_S\) depends only
on the fixed maps \(S\) and \(S^{-1}\).
\end{proposition}

\begin{proof}
On the left cone, \Cref{lem:scalar-cone-flip} gives
\(\Phi_K(x-\frac12,u_-)=u_-\), so
\(\mathcal C_K^{\mathscr J}(x,u)=u\).
On the right cone, it gives
\(\Phi_K(x-\frac12,u_-)=-u_-\), and
\eqref{eq:involution-diagonalization} yields
\(\mathcal C_K^{\mathscr J}(x,u)=\mathscr Ju\).
The realization and coefficient bounds follow from the fixed linear maps
\(S,S^{-1}\) and the coordinatewise scalar cone flips.
\end{proof}

For the universal doubled fiber, take \(Y=\widehat E\) and
\(\mathscr J=\mathcal J\).  A convenient diagonalizing map is
\[
  S
  \begin{bmatrix}u\\v\end{bmatrix}
  =
  \begin{bmatrix}
    \frac12(u+v)\\[1mm]
    \frac12(u-v)
  \end{bmatrix}.
\]
The corresponding projections are
\[
  P_+=\frac12(I+\mathcal J),
  \qquad
  P_-=\frac12(I-\mathcal J),
\]
and \eqref{eq:general-involution-switch} becomes
\begin{equation}
\label{eq:swap-cone-switch}
  \mathcal C_K(x,u)
  =
  P_+u+
  \Phi_K(x-\textstyle\frac12,P_-u),
\end{equation}
where the scalar flip is applied coordinatewise.  
Since \(\phi_K(s,\cdot)\) is odd, the coordinatewise map
\(\Phi_K(s,\cdot)\) preserves the swap-odd subspace.
We use this
swap-specialized notation in the remainder of the main construction.

\subsection{Linear control of the odd component}

For a Lipschitz map \(H:[0,1]\to\widehat E\), write \(\Lip(H)\) for its Lipschitz constant in the
max norm.

\begin{lemma}[Fold-cone estimate]
\label{lem:fold-cone-estimate}
Let \(G:[0,1]\to E\) be continuous and satisfy the seam identity
\(T_0G(1)=T_1G(0)\).  Set \(G^{\#}(x)=(G(x),G(1-x))\) and
\[
H(x)=\widehat T G^{\#}(\tau(x)).
\]
Then we have
\begin{equation}
\label{eq:odd-cone-bound}
\|P_-H(x)\|_\infty
\le
2\|P_-\widehat T\|\,\Lip(G^{\#})
\textstyle|x-\frac12| ,
\qquad x\in[0,1].
\end{equation}
\end{lemma}

\begin{proof}
By the seam identity, \(H(\frac12)=\widehat T G^{\#}(1)\) is fixed by \(\mathcal J\), so
\(P_-H(\frac12)=0\).  Therefore
\[
\|P_-H(x)\|_\infty
\le
\|P_-\widehat T\|\,
\|G^{\#}(\tau(x))-G^{\#}(1)\|_\infty.
\]
Since \(|\tau(x)-1|=2|x-\frac12|\), the claimed estimate follows.
\end{proof}

\begin{corollary}[Exact folded update]
\label{cor:exact-folded-update}
We have
\begin{equation}
\label{eq:folded-update-switch}
\mathcal C_{K(G)}\bigl(x,\widehat T G^{\#}(\tau(x))\bigr)
=
\begin{cases}
\widehat T G^{\#}(\tau(x)),&x\le\frac12,\\
\mathcal J\widehat T G^{\#}(\tau(x)),&x\ge\frac12 ,
\end{cases}
\end{equation}
with
\(K(G)=1+2\|P_-\widehat T\|\,\Lip(G^{\#})\).
\end{corollary}

\subsection{Lipschitz growth}

Equip the doubled state space with the maximum product norm, and use the
corresponding induced operator norms.  Then
\[
  \Lip(G_m^{\#})=\Lip(G_m)=:L_m,
  \qquad
  \Lip(b^{\#})=\Lip(b)=:L_b,
\]
and
\[
  T_*:=\max\{\|T_0\|,\|T_1\|\}
  =\|\widehat T\|
  =\|\mathcal J\widehat T\|.
\]
Since \(\tau\) has slope of magnitude \(2\) on each half interval,
\eqref{eq:folded-recursion} shows that each branch of \(G_{m+1}^{\#}\)
has Lipschitz constant at most \(2T_*L_m+L_b\).  The two branches meet
continuously at \(x=\frac12\) by \eqref{eq:fold-agreement}, and therefore
\begin{equation}
\label{eq:Lip-recursion}
  L_{m+1}\le 2T_*L_m+L_b .
\end{equation}

\begin{lemma}[Exponential Lipschitz bound]
\label{lem:Lip-bound}
There exist constants \(C_L>0\) and \(\Lambda\ge1\), depending only on the
fixed mask, initial profile, and forcing term, such that
\begin{equation}
\label{eq:Lip-exponential}
  \Lip(G_m^{\#})\le C_L\Lambda^m,
  \qquad m\ge0.
\end{equation}
Consequently, the cone scales
\begin{equation}
\label{eq:Km-definition}
  K_m
  =
  1+ 2\|P_-\widehat T\|\,\Lip(G_m^{\#})
\end{equation}
satisfy \(K_m\le C_K\Lambda^m\) for some \(C_K>0\).
\end{lemma}

\begin{proof}
Set \(\Lambda:=1+2T_*\) and \(C_L:=L_0+L_b\).  The recursion
\eqref{eq:Lip-recursion} then gives, by induction,
\[
  L_m\le C_L\Lambda^m.
\]
Since \(L_m=\Lip(G_m^{\#})\), this proves \eqref{eq:Lip-exponential}.
The estimate for \(K_m\) follows with
\(C_K:=1+2\|P_-\widehat T\|\,C_L\).
\end{proof}

\section{Folded replay by the residual memory controller}
\label{sec:memory}

The memory controller used in this section is the binary specialization of the residual memory
construction in \cite[\S3.2]{affine-memory}.  We recall the specialization needed below in order
to keep the folded argument self-contained. 
The additional point here is its folded interpretation: the polygonal
loop is read through the reflection quotient of circle doubling, so a
single global affine readout recovers the tent-map orbit.

\subsection{Circle lift and folded readout}

Let
\(\T:=\R/\Z\),
\(d([t]):=[2t]\),
and let \(\emb:\T\to\Gamma\subset[-1,1]^2\) be the diamond embedding determined by
\begin{equation}
\label{eq:diamond-vertices}
  \emb([0])=(1,0),\qquad
  \emb([\textstyle\frac14])=(0,1),\qquad
  \emb([\textstyle\frac12])=(-1,0),\qquad
  \emb([\textstyle\frac34])=(0,-1),
\end{equation}
with affine interpolation on each quarter of the circle.  Then
\begin{equation}
\label{eq:antipodal-E}
\textstyle
  \emb([t+\frac12])=-\emb([t]),
  \qquad
  \frac12\le \|\emb([t])\|_\infty\le1.
\end{equation}
Circle doubling induces a CPwL map \(F_\Gamma:\Gamma\to\Gamma\) given by
\begin{equation}
\label{eq:F-Gamma}
  F_\Gamma(\emb([t]))=\emb(d([t]))=\emb([2t]).
\end{equation}
We fix a global CPwL extension \(F:\R^2\to\R^2\) of \(F_\Gamma\).

Define the folded readout \(\Pi:\R^2\to\R\) by
\(\Pi(x,y)=(1-x)/2\).
By the explicit parametrization of the diamond,
we have
\(\Pi(\emb([t]))=2\operatorname{dist}(t,\Z)\),
and hence
\[
  \Pi\bigl(F(\emb([t]))\bigr)
  =
  2\operatorname{dist}(2t,\Z)
  =
  \tau\bigl(\Pi(\emb([t]))\bigr).
\]
Since \(\Pi(\emb([x/2]))=x\) for \(x\in[0,1]\), iteration yields
\begin{equation}
\label{eq:tent-from-circle}
  \Pi\bigl(F^j(\emb([x/2]))\bigr)
  =
  \tau^j(x),
  \qquad
  x\in[0,1],\quad j\ge0.
\end{equation}

\subsection{Injective memory lift}

Set
\(\mem=[-1,1]^2\) and
\(X=\Gamma\times\mem\subset\R^4\).
Choose \(\alpha,\beta>0\) satisfying
\begin{equation}
\label{eq:alpha-beta}
  \alpha+\beta\le1,
  \qquad
  2\alpha<\beta,
\end{equation}
and define
\begin{equation}
\label{eq:memory-Q}
  \mathcal Q(r,y)
  :=
  \bigl(F(r),\beta r+\alpha y\bigr),
  \qquad
  (r,y)\in\R^2\times\R^2.
\end{equation}

The following result is the binary specialization of the residual memory
construction proved in \cite[\S3.2]{affine-memory}.  We include the short
argument for completeness and to fix the notation used below.

\begin{lemma}[Injective folded-loop memory]
\label{lem:memory-injective}
The map \(\mathcal Q\) sends \(X\) into itself and is injective on \(X\).
Its inverse on \(\mathcal Q(X)\) is CPwL and admits a global CPwL extension
\(\mathcal P:\R^4\to\R^4\) satisfying
\begin{equation}
\label{eq:PQ-identity}
  \mathcal P(\mathcal Q(Z))=Z,
  \qquad Z\in X.
\end{equation}
\end{lemma}

\begin{proof}
For \((r,y)\in X\), we have \(F(r)\in\Gamma\) and
\[
  \|\beta r+\alpha y\|_\infty
  \le
  \beta\|r\|_\infty+\alpha\|y\|_\infty
  \le
  \beta+\alpha
  \le1.
\]
Thus \(\mathcal Q(X)\subset X\).

Suppose that
\[
  \mathcal Q(\emb([t]),y)
  =
  \mathcal Q(\emb([u]),\tilde y).
\]
Since \(F(\emb([v]))=\emb([2v])\) on \(\Gamma\), equality of the first
components gives
\(\emb([2t])=\emb([2u])\).
The injectivity of \(\emb:\T\to\Gamma\) implies
\([2t]=[2u]\), and hence either
\([u]=[t]\) or
\([u]=[t+\frac12]\).
In the first case, equality of the second components gives
\(y=\tilde y\).  In the second case,
\eqref{eq:antipodal-E} gives
\[
  2\beta \emb([t])
  =
  \alpha(\tilde y-y).
\]
The left-hand side has max norm at least \(\beta\), whereas the
right-hand side has max norm at most \(2\alpha\), contradicting
\eqref{eq:alpha-beta}.  Thus the second case is impossible, and
\(\mathcal Q|_X\) is injective.

Finally, \(F|_\Gamma\) is CPwL, so after a finite subdivision
\(\mathcal Q\) is affine on every product cell of \(X\).  Its
injectivity implies that the image cells form a finite polyhedral
complex and that the inverse affine formulas agree on their
intersections.  They therefore define a CPwL inverse on
\(\mathcal Q(X)\), which admits a global CPwL extension
\(\mathcal P:\R^4\to\R^4\), as in
\cite[\S3.2]{affine-memory}.
\end{proof}

\subsection{Exact backward folded replay}

We extend the folded readout to the full controller space, using the same
symbol, by
\[
  \Pi(r,y):=\Pi(r),
  \qquad (r,y)\in\R^2\times\R^2.
\]
Thus \(\Pi:\R^4\to\R\) is a fixed affine map that ignores the memory
coordinate.

For \(x\in[0,1]\), define
\begin{equation}
\label{eq:Z0}
  Z_0(x)
  :=
  \bigl(\emb([x/2]),0\bigr),
  \qquad
  Z_j(x)
  :=
  \mathcal Q^j(Z_0(x)).
\end{equation}
The loop component of \(Z_j(x)\) is
\(F^j(\emb([x/2])) = \emb(d^j([x/2]))\).

\begin{lemma}[Backward folded replay]
\label{lem:backward-replay}
For \(0\le j\le n\), one has
\begin{equation}
\label{eq:backward-replay}
  \mathcal P^{\,n-j}(Z_n(x))
  =
  Z_j(x).
\end{equation}
Moreover, we have
\begin{equation}
\label{eq:folded-readout}
  \Pi(Z_j(x))
  =
  \tau^j(x).
\end{equation}
\end{lemma}

\begin{proof}
Since \(Z_{k+1}(x)=\mathcal Q(Z_k(x))\), identity
\eqref{eq:PQ-identity} gives
\(\mathcal P(Z_{k+1}(x))=Z_k(x)\).  Iteration proves
\eqref{eq:backward-replay}, while \eqref{eq:folded-readout} follows from
\eqref{eq:tent-from-circle}.
\end{proof}

The memory coordinate retains the information lost under circle doubling,
whereas the folded residual is read directly from the loop component.  On the
circle, the two inverse branches are represented by uniformly separated
antipodal points, which makes the injective skew-product possible.

\section{Unit-interval realization of the affine cascade}
\label{sec:unit-realization}

We now combine reflection doubling, cone switching, and backward folded replay.  Throughout this
section, \(G_m=\vect(W^m\gamma)\), \(G_m^{\#}\) is its reflected double, and \(b^{\#}\) is the doubled
forcing.

\subsection{The nested folded recursion}

By \Cref{prop:folded-recursion,cor:exact-folded-update}, the exact functional recursion is
\begin{equation}
\label{eq:functional-folded-recursion}
G_{m+1}^{\#}(x)
=
\mathcal C_{K_m}
\bigl(x,\widehat T G_m^{\#}(\tau(x))\bigr)
+b^{\#}(x),
\end{equation}
where \(K_m\) is chosen as in \eqref{eq:Km-definition}.

For a fixed input \(x\), set
\[
x_j=\tau^j(x),
\qquad j=0,\ldots,n.
\]
The nested evaluation of \eqref{eq:functional-folded-recursion} starts from
\(G_0^{\#}(x_n)\).  It then applies the affine stages along the residuals in the order
\(x_{n-1},x_{n-2},\ldots,x_0\).
Define
\begin{equation}
\label{eq:backward-U-init}
U_n(x)=G_0^{\#}(x_n),
\end{equation}
and, for \(j=n-1,n-2,\ldots,0\), define
\begin{equation}
\label{eq:backward-U}
U_j(x)
=
\mathcal C_{K_{n-1-j}}
\bigl(x_j,\widehat T U_{j+1}(x)\bigr)
+b^{\#}(x_j).
\end{equation}

\begin{lemma}[Backward folded evaluation]
\label{lem:backward-folded-evaluation}
For \(j=0,\ldots,n\), we have
\begin{equation}
\label{eq:U-correctness}
U_j(x)=G_{n-j}^{\#}(x_j).
\end{equation}
In particular, it holds that
\begin{equation}
\label{eq:U0-output}
U_0(x)=G_n^{\#}(x).
\end{equation}
\end{lemma}

\begin{proof}
The statement is true for \(j=n\) by definition.  Suppose it holds for \(j+1\).  Set
\(m=n-1-j\).  Then \(U_{j+1}=G_m^{\#}(x_{j+1})\), and
\(x_{j+1}=\tau(x_j)\).  The exact functional recursion
\eqref{eq:functional-folded-recursion} gives
\[
U_j
=
G_{m+1}^{\#}(x_j)
=
G_{n-j}^{\#}(x_j).
\]
Backward induction proves the claim.
\end{proof}

\subsection{Network architecture}

For a fixed depth \(n\), the unit-interval network is assembled as follows.

\begin{enumerate}[label=\textup{(\arabic*)},leftmargin=2.8em]
  \item From the input \(x\), compute the valid initial memory state
  \(Z_0(x)=(\emb([x/2]),0)\).

  \item Apply \(n\) copies of the fixed CPwL memory update \(\mathcal Q\) to obtain \(Z_n(x)\).

  \item Read \(x_n=\Pi(Z_n(x))\), and initialize
  \(U_n=G_0^{\#}(x_n)\).  This is a fixed CPwL readout because \(G_0\) is fixed and CPwL.

  \item For each \(m=0,\ldots,n-1\), apply one backward block.  The block first applies
  \(\mathcal P\) to recover the preceding memory state, reads the corresponding folded residual
  \(x_{n-1-m}\), and performs the update \eqref{eq:backward-U} using the cone scale \(K_m\).

  \item After the last backward block, project \(U_0\) onto its first \(E\)-component.  This is
  \(G_n(x)=\vect(W^n\gamma)(x)\).
\end{enumerate}

All modules except the stage-indexed cone scales are fixed CPwL maps.  The scale \(K_m\) is a
fixed numerical parameter of the \(m\)-th block, not a network variable.

\begin{theorem}[Unit-interval affine realization]
\label{thm:unit-affine}
Let \(W\) be as in \eqref{eq:intro-W}, with compactly supported CPwL input and forcing in a
preserved support window \([0,L]\).  There exist integers \(C_0,C_1>0\), independent of \(n\),
such that
\begin{equation}
\label{eq:unit-affine-class}
\vect(W^n\gamma)
\in
\Ups_{C_0,C_1n}(\ReLU;1,pL)
\end{equation}
on \([0,1]\).  The weights and biases may be bounded by \(C_2\Lambda^n\) for constants depending
only on the fixed mask, support window, and CPwL data.
\end{theorem}

\begin{proof}
The forward memory pass consists of \(n\) copies of one fixed CPwL module, and hence has fixed
width and depth \(O(n)\).  The terminal readout has fixed complexity.  Each backward block consists
of the fixed inverse memory module, the fixed folded readout, one fixed linear map \(\widehat T\),
the cone switch \(\mathcal C_{K_m}\), and the fixed forcing readout \(b^{\#}\).  It therefore has fixed
depth and width depending only on \(pL\).  There are \(n\) backward blocks.

Exactness follows from \Cref{lem:backward-replay,lem:backward-folded-evaluation}.  By
\Cref{lem:Lip-bound}, \(K_m\le C\Lambda^m\), and the cone-switch weights are bounded by
\(C'(1+K_m)\).  The remaining modules have fixed coefficients.  Enlarging the constants yields the
coarse bound \(C_2\Lambda^n\).
\end{proof}

\subsection{The homogeneous case}

Setting \(B\equiv0\) yields a folded-memory proof of the binary homogeneous
vector-valued theorem.  Memory is not intrinsically necessary in this case:
besides the earlier forward constructions, the cone switch also admits an
adjoint homogeneous realization.  
For endpoint-zero reference atoms, the swap-odd adjoint candidate satisfies
the same cone estimate, so the adjoint cascade can be propagated forward
without residual memory; see Appendix~\ref{app:adjoint-homogeneous}.

As in the earlier affine constructions, we use a primal cascade with
residual memory for the affine problem and retain the same formulation in
the homogeneous case.  This gives a unified argument, displays the
reflection-folding mechanism directly, and avoids an atomic decomposition:
the construction propagates the complete vectorized initial profile and,
in the affine case, the complete forcing profile.

\section{Global realization and stage-dependent forcing}
\label{sec:global}

\subsection{Globalization by clamped gluing}

We use the standard clamped-gluing device employed in
\cite{source,loop} to pass from the vectorized realization on one unit
interval to the global compactly supported function.  We recall the
formula and its short verification for completeness.

Define
\begin{equation}
\label{eq:clip}
  \clip(s):=\sigma(s)-\sigma(s-1).
\end{equation}
Thus \(\clip(s)=0\) for \(s\le0\), \(\clip(s)=s\) for
\(0\le s\le1\), and \(\clip(s)=1\) for \(s\ge1\).

\begin{lemma}[Clamped gluing]
\label{lem:clamped-gluing}
Let \(g_0,\ldots,g_{L-1}:[0,1]\to\R^p\) be continuous and satisfy
\begin{equation}
\label{eq:gluing-compatibility}
  g_{k-1}(1)=g_k(0),
  \qquad k=1,\ldots,L-1,
\end{equation}
together with \(g_0(0)=g_{L-1}(1)=0\).  Then
\begin{equation}
\label{eq:gluing-formula}
  \mathcal G(t)
  :=
  g_0(\clip(t))
  +
  \sum_{k=1}^{L-1}
  \bigl[g_k(\clip(t-k))-g_k(0)\bigr]
\end{equation}
satisfies
\[
  \mathcal G(t)=g_m(t-m),
  \qquad t\in[m,m+1],
\]
for \(m=0,\ldots,L-1\), and vanishes outside \([0,L]\).
\end{lemma}

\begin{proof}
For \(t\in[m,m+1]\), the clamped arguments equal \(1\) for
\(k<m\), \(t-m\) for \(k=m\), and \(0\) for \(k>m\).
The resulting sum telescopes by \eqref{eq:gluing-compatibility} to
\(g_m(t-m)\).  The same calculation, using the endpoint conditions,
gives \(\mathcal G(t)=0\) for \(t\notin[0,L]\).
\end{proof}

If a fixed-width, depth-\(O(n)\) network returns the complete
vectorization on \([0,1]\), then the finitely many shifted and clamped
copies required in \eqref{eq:gluing-formula} can be evaluated in
parallel.  Since \(L\) is fixed, globalization changes only the constant
width factor and preserves depth \(O(n)\).

\subsection{Main binary affine theorem}

Combining the unit-interval construction with clamped gluing yields the global realization theorem.

\begin{theorem}[Exact binary affine realization]
\label{thm:main-binary-affine}
Let
\begin{equation}
\label{eq:main-W}
  (W\gamma)(t)
  =
  \sum_{j\in\Z}A_j\gamma(2t-j)+B(t),
\end{equation}
where only finitely many matrices \(A_j\in\R^{p\times p}\) are nonzero.
Assume that the homogeneous part preserves a support window \([0,L]\),
and let \(\gamma,B:\R\to\R^p\) be CPwL and supported in that window.
Then there exist constants \(C_0,C_1>0\), independent of \(n\), such that
\begin{equation}
\label{eq:main-theorem-class}
  W^n\gamma
  \in
  \Ups_{C_0,C_1n}(\ReLU;1,p),
  \qquad n\ge1.
\end{equation}
Moreover, the realizing networks may be chosen with weights and biases
bounded by \(C_2\Lambda^n\), where \(C_2,\Lambda>0\) depend only on the
fixed refinement and CPwL data.
\end{theorem}

\begin{proof}
By \Cref{thm:unit-affine}, the complete vectorization of \(W^n\gamma\)
on \([0,1]\) has an exact fixed-width, depth-\(O(n)\) realization.
Its coordinate functions satisfy
\eqref{eq:gluing-compatibility} and the outer endpoint conditions,
because they are the unit-cell restrictions of the continuous function
\(W^n\gamma\), supported in \([0,L]\).  Applying
\Cref{lem:clamped-gluing} gives the global realization.  The parallel
evaluation of the finitely many shifted copies changes the width and
coefficient bounds only by fixed factors.
\end{proof}

No symmetry or seam-separation hypothesis is imposed on the forcing
term.  The construction keeps its full vectorization
\(b(x)=\vect B(x)\) and inserts the reflected pair
\(b^\#(x)=(b(x),b(1-x))^{\trans}\) directly into the folded affine
recursion.  Thus both residual orientations are carried simultaneously,
without decomposing \(B\) into translated reference atoms or invoking
translation covariance.  The compatibility required by the cone switch
is instead the cell-seam identity of the complete vectorized profile,
which is preserved by the affine recursion.

\subsection{Stage-dependent forcing}

Consider the nonstationary affine recursion
\begin{equation}
\label{eq:stage-dependent-recursion}
  \gamma_{r+1}=V\gamma_r+B_r,
  \qquad r=0,\ldots,n-1,
\end{equation}
with a fixed homogeneous binary refinement operator \(V\).

\begin{corollary}[Fixed-span stage-dependent forcing]
\label{cor:stage-dependent}
Let \(\gamma_0:\R\to\R^p\) and
\(B^{(0)},\ldots,B^{(N)}:\R\to\R^p\) be fixed CPwL functions supported
in the preserved support window.  Suppose that
\begin{equation}
\label{eq:fixed-span}
  B_r
  =
  \sum_{\alpha=0}^{N}
  \lambda_{r,\alpha}B^{(\alpha)},
  \qquad r=0,\ldots,n-1.
\end{equation}
Then the output \(\gamma_n\) of
\eqref{eq:stage-dependent-recursion} admits an exact fixed-width,
depth-\(O(n)\) ReLU realization.

More precisely, if
\[
  M_n
  :=
  \max_{\substack{0\le r<n\\0\le\alpha\le N}}
  |\lambda_{r,\alpha}|,
\]
then the realizing network may be chosen with weights and biases bounded by
\(C(1+M_n)\Lambda^n\),
where \(C,\Lambda>0\) depend only on the mask, the support window,
\(\gamma_0\), and the forcing templates.
\end{corollary}

\begin{proof}
Write
\[
  G_r:=\vect\gamma_r,
  \qquad
  b^{(\alpha)}:=\vect B^{(\alpha)},
  \qquad
  b_r:=\sum_{\alpha=0}^{N}\lambda_{r,\alpha}b^{(\alpha)}.
\]
The one-step folded identity remains valid with the stage-dependent forcing:
\begin{equation}
\label{eq:stage-dependent-folded}
  G_{r+1}^{\#}(x)
  =
  \begin{cases}
    \widehat T G_r^{\#}(\tau(x))+b_r^{\#}(x),
      &0\le x\le\frac12,\\[1mm]
    \mathcal J\widehat T G_r^{\#}(\tau(x))+b_r^{\#}(x),
      &\frac12\le x\le1.
  \end{cases}
\end{equation}

Fix \(x\in[0,1]\) and set \(x_j:=\tau^j(x)\).  For each stage \(r\), let
\(K_r\) be defined by \eqref{eq:Km-definition}, with \(G_m^\#\) there
replaced by the stage-dependent profile \(G_r^\#\).  The residual memory
controller supplies the folded residuals in the reverse order
\(x_n,x_{n-1},\ldots,x_0\).
Initialize
\[
  U_n:=G_0^\#(x_n),
\]
and, for \(j=n-1,\ldots,0\), define
\begin{equation}
\label{eq:stage-dependent-backward}
  U_j
  :=
  \mathcal C_{K_{n-1-j}}
  \bigl(x_j,\widehat T U_{j+1}\bigr)
  +
  b_{n-1-j}^{\#}(x_j).
\end{equation}
A backward induction using \eqref{eq:stage-dependent-folded} at stage
\(r=n-j-1\) yields
\begin{equation}
\label{eq:stage-dependent-backward-invariant}
  U_j=G_{n-j}^{\#}(x_j),
  \qquad j=0,\ldots,n.
\end{equation}
In particular, we have \(U_0=G_n^\#(x)\).

The finitely many template readouts \(b^{(\alpha)\#}\) are evaluated in
parallel, and at stage \(r\) their affine combination with coefficients
\(\lambda_{r,\alpha}\) produces \(b_r^\#\).  Since the number of templates
is fixed, each backward block has fixed width and depth.  The memory
controller followed by the \(n\) backward blocks therefore realizes
\(G_n=\vect\gamma_n\) on \([0,1]\) with fixed width and depth \(O(n)\).
The global realization of \(\gamma_n\) follows from
\Cref{lem:clamped-gluing}.

It remains to track the coefficient sizes.  Let
\(L_r:=\Lip(G_r^\#)\).  From \eqref{eq:stage-dependent-folded},
\[
  L_{r+1}
  \le
  2T_*L_r+\Lip(b_r^\#).
\]
Moreover, we infer
\[
  \Lip(b_r^\#)
  \le
  \sum_{\alpha=0}^{N}
  |\lambda_{r,\alpha}|\,
  \Lip\bigl(b^{(\alpha)\#}\bigr)
  \le
  C_BM_n,
\]
where \(C_B\) depends only on the fixed templates.  Hence, after enlarging
the constants if necessary, we conclude
\(L_r
  \le
  C(1+M_n)\Lambda^r,
  \qquad 0\le r\le n\).
The same estimate holds for \(K_r\).  All other modules have fixed
coefficients, while the template combinations use coefficients bounded by
\(M_n\), yielding the asserted bound.
\end{proof}

\section{Reflection-equivariant reduction}
\label{sec:reflection-equivariant}

The reflected double applies to every binary mask.  If the original refinement
rule is itself reflection equivariant, however, the doubled cascade lies in an
invariant subspace of half the dimension.

Let \(C\in\R^{p\times p}\) be an involution and define
\begin{equation}
\label{eq:physical-reflection}
  (\mathcal Rf)(t):=Cf(L-t).
\end{equation}
On the vectorization space \(E=(\R^p)^L\), define the corresponding involution
\begin{equation}
\label{eq:fiber-reflection}
  (Ju)_k:=Cu_{L-1-k},
  \qquad k=0,\ldots,L-1.
\end{equation}

\begin{lemma}[Reflection equivariance]
\label{lem:reflection-equivariance}
Suppose that
\begin{equation}
\label{eq:mask-reflection}
  A_{L-j}=CA_jC,
  \qquad j\in\Z.
\end{equation}
Then we have
\[
  V\mathcal R=\mathcal RV
  \qquad\text{and}\qquad
  T_1J=JT_0.
\]
\end{lemma}

\begin{proof}
For the first identity, one has
\[
  (V\mathcal Rf)(t)
  =
  \sum_j A_jC f(L-2t+j),
\]
whereas the change of index \(j\mapsto L-j\) gives
\[
  (\mathcal RVf)(t)
  =
  \sum_j CA_{L-j}f(L-2t+j).
\]
These expressions agree by \eqref{eq:mask-reflection}.

For the transition matrices, using
\[
  (T_q u)_k
  =
  \sum_{\ell=0}^{L-1}A_{q+2k-\ell}u_\ell
\]
and setting \(r=L-1-\ell\), we obtain
\[
  (T_1Ju)_k
  =
  \sum_{r=0}^{L-1}
  A_{2+2k-L+r}C\,u_r
  =
  \sum_{r=0}^{L-1}
  CA_{2L-2-2k-r}u_r
  =
  (JT_0u)_k,
\]
again by \eqref{eq:mask-reflection}.
\end{proof}

Let \(\eta\in\{1,-1\}\), and suppose that the initial profile and forcing
have the same reflection parity,
\[
  \mathcal R\gamma_0=\eta\gamma_0,
  \qquad
  \mathcal RB=\eta B.
\]
Since \(V\) commutes with \(\mathcal R\), every affine iterate has the same
parity.  In vectorized form, it is
\begin{equation}
\label{eq:parity}
  G_m(1-x)=\eta JG_m(x),
  \qquad
  b(1-x)=\eta Jb(x).
\end{equation}
Consequently, the reflected doubles take values in the twisted diagonal
\[
  \Delta_\eta
  :=
  \left\{
    \begin{bmatrix}u\\ \eta Ju\end{bmatrix}
    :u\in E
  \right\}
  \subset E\oplus E.
\]
More precisely, with
\[
  \iota_\eta:E\to\Delta_\eta,
  \qquad
  \iota_\eta u
  :=
  \begin{bmatrix}u\\ \eta Ju\end{bmatrix},
\]
one has
\[
  G_m^\#(x)=\iota_\eta G_m(x),
  \qquad
  b^\#(x)=\iota_\eta b(x).
\]

The relation \(T_1J=JT_0\) shows that \(\Delta_\eta\) is invariant under
the folded update.  Indeed,
\[
  \widehat T\,\iota_\eta u
  =
  \iota_\eta(T_0u),
  \qquad
  \mathcal J\widehat T\,\iota_\eta u
  =
  \iota_\eta(\eta JT_0u).
\]
Identifying \(\Delta_\eta\) with \(E\) through \(\iota_\eta\), the folded
recursion therefore reduces to
\begin{equation}
\label{eq:parity-folded-recursion}
  G_{m+1}(x)
  =
  \begin{cases}
    T_0G_m(\tau(x))+b(x),
      &0\le x\le\frac12,\\[1mm]
    \eta JT_0G_m(\tau(x))+b(x),
      &\frac12\le x\le1.
  \end{cases}
\end{equation}
The two candidates agree at \(x=\frac12\).  Indeed, the cell-seam identity,
\eqref{eq:parity}, and \(T_1J=JT_0\) give
\[
  T_0G_m(1)
  =
  T_1G_m(0)
  =
  \eta JT_0G_m(1).
\]
Set
\(Y_m(x):=T_0G_m(\tau(x))\).
Choose a fixed diagonalizing map \(S_\eta\) for the involution \(\eta J\),
as in \Cref{prop:involution-cone-switch}, and write
\(S_\eta u=(u_+,u_-)\).
Denote by \(\tilde S_{\eta}u:=u_-\) its odd-coordinate component.  The fold
agreement proved above gives
\((\eta J)Y_m(\frac12) = Y_m(\frac12)\),
and hence
\(\tilde S_{\eta}Y_m(\frac12)=0\).
Since the tent map is \(2\)-Lipschitz, we have
\[
  \|\tilde S_{\eta}Y_m(x)\|_\infty
  \le
  2\|\tilde S_{\eta}T_0\|\,\Lip(G_m)
  |x-\textstyle\frac12|.
\]
Thus with
\(K_m = 1+2\|\tilde S_{\eta}T_0\|\,\Lip(G_m)\),
\Cref{prop:involution-cone-switch} realizes the two branches of
\eqref{eq:parity-folded-recursion} directly on \(E\).  Combining this
reduced recursion with the same residual memory controller and
clamped-gluing step as before yields the following result.

\begin{corollary}[Parity-reduced realization]
\label{cor:parity-reduced}
Under \eqref{eq:mask-reflection}, suppose that the initial profile and forcing
have the same reflection parity \(\eta\).  Then the exact fixed-width,
depth-\(O(n)\) realization may be constructed with a \(pL\)-dimensional
cascade fiber instead of the \(2pL\)-dimensional doubled fiber.
\end{corollary}

\begin{example}[The L\'evy--C dragon and its anchored defect]
\label{ex:dragon-parity}
The dragon generators and anchored-profile construction were developed
in \cite{loop}.  Here we record the additional reflection-parity structure
of the L\'evy--C mask.
Let \(p=2\), \(L=1\), and set
\[
  C=
  \begin{bmatrix}
    1&0\\
    0&-1
  \end{bmatrix},
  \qquad
  A_0
  =
  \frac12
  \begin{bmatrix}
    1&-1\\
    1& 1
  \end{bmatrix},
  \qquad
  A_1
  =
  \frac12
  \begin{bmatrix}
    1& 1\\
   -1& 1
  \end{bmatrix},
\]
with \(A_j=0\) for \(j\notin\{0,1\}\).
The matrices \(A_0\) and \(A_1\) are the two similarities of the
L\'evy--C dragon \cite{Levy,DK}.  They satisfy \(A_1=CA_0C\), so the
mask is reflection equivariant.  Since its support is \(\{0,1\}\), the
corresponding refinement operator preserves \([0,1]\).

Here \(E=(\R^2)^1\cong\R^2\), and the fiber involution \(J\) is identified
with \(C\).  Hence a profile of parity \(\eta\) has its reflected state in
the twisted diagonal
\[
  \Delta_\eta
  =
  \left\{
    \begin{bmatrix}u\\ \eta Cu\end{bmatrix}
    :u\in\R^2
  \right\}.
\]
By \Cref{cor:parity-reduced}, the parity-reduced cascade is therefore
propagated on a two-dimensional fiber instead of the universal four-dimensional doubled fiber.

Let \(\e_1,\e_2\) denote the standard basis of \(\R^2\), and define the
endpoint-extended anchor profile
\[
  \gamma_{\rm a}(t)
  :=
  \begin{cases}
    0,&t\le0,\\
    t\e_1,&0\le t\le1,\\
    \e_1,&t\ge1.
  \end{cases}
\]
Since \(A_0+A_1=I_2\), the anchor mismatch
\(B_{\rm a}:=V\gamma_{\rm a}-\gamma_{\rm a}\) is compactly supported.
A direct calculation gives
\[
  B_{\rm a}(t)
  =
  \begin{cases}
    t\e_2,&0\le t\le\frac12,\\[1mm]
    (1-t)\e_2,&\frac12\le t\le1,\\[1mm]
    0,&t\notin[0,1].
  \end{cases}
\]
In particular,
\(B_{\rm a}(1-t)=-CB_{\rm a}(t)\), so \(B_{\rm a}\) has parity
\(\eta=-1\).

If \(\gamma_0=\gamma_{\rm a}\) and
\(\gamma_{m+1}=V\gamma_m\), then the compactly supported defects
\(\xi_m:=\gamma_m-\gamma_{\rm a}\) satisfy
\[
  \xi_0=0,
  \qquad
  \xi_{m+1}=V\xi_m+B_{\rm a}.
\]
Hence the endpoint-extended L\'evy--C approximants are recovered from the
parity-reduced affine recursion in the \(\eta=-1\) sector, followed by
addition of the fixed anchor \(\gamma_{\rm a}\).
\end{example}

\section{Conclusions}
\label{sec:conclusions}

We have proved an exact fixed-width, depth-\(O(n)\) ReLU realization theorem for binary
vector-valued affine refinement iterates with compactly supported CPwL initial data and forcing.
The construction removes the seam-separation requirement of the earlier offset-frame approach and
also extends to stage-dependent forcing drawn from a fixed finite-dimensional family.

The main new mechanism is universal reflection doubling.  Pairing each residual profile with its
reflection replaces the two binary transition matrices by one block matrix together with a fixed swap
involution.  The cell-seam identity makes the swap-odd component vanish at the tent fold, while its
Lipschitz control places the two branch candidates in a quantitative cone.  A fixed CPwL cone switch
can therefore select the correct branch exactly, without multiplying the evolving state by a variable
selector.

Because the primal recursion must be evaluated in reverse residual order,
we combine the folded cascade with the residual memory controller of
\cite{affine-memory}.  Reflection doubling and cone switching resolve the
binary branch ambiguity, while the injective loop-memory lift supplies exact
backward replay.  For reflection-equivariant masks, the doubled cascade
reduces to a parity sector.  Appendix~\ref{app:adjoint-homogeneous} shows
that memory can be avoided in the homogeneous case by a forward adjoint
construction for endpoint-zero atoms.

Appendix~\ref{app:fan-switch} identifies the binary cone switch as the
two-branch case of a fixed-arity fan switch.  The \(M\)-ary transition
matrices factor through one common polyphase operator preceded by
digit-dependent shifts and reflections, and adjacent candidates agree at the
zigzag folds.  Thus the folded-memory architecture extends to every fixed
arity without offset frames or a global symmetry relating all branches.
Sharper coefficient bounds, finite-precision stability, and multidimensional
affine refinement remain open directions.

\appendix

\section{Adjoint cone switching for homogeneous atoms}
\label{app:adjoint-homogeneous}

This appendix gives a memory-free forward adjoint realization for
homogeneous refinement.  For an endpoint-zero unit-interval atom, the
terminal scalar factor is incorporated into the initial adjoint state.
Rerunning the residual orbit forward then propagates the adjoint cascade,
without backward replay or a memory coordinate.

\subsection{Adjoint recursion and cone estimate}

Let \(h:[0,1]\to\R\) be CPwL with \(h(0)=h(1)=0\), extended by zero to
\(\R\), and let \(v\in\R^p\).  Define
\[
  \gamma_{h,v}(t):=h(t)v,
  \qquad
  u:=(v,0,\ldots,0)\in E.
\]
Then \(G_0=\vect(\gamma_{h,v})\) satisfies \(G_0(x)=h(x)u\) on
\([0,1]\).  Set
\[
  G_n:=\vect(V^n\gamma_{h,v}),
  \qquad
  \widehat u:=
  \begin{bmatrix}
    u\\u
  \end{bmatrix}
  \in\widehat E.
\]

We identify \(E^*\) and \(\widehat E^*\) with their Euclidean coordinate
spaces, so the adjoints of \(T_q\) and \(\widehat T\) are their
transposes and the swap involution is self-adjoint.  For
\(\psi\in E^*\), write
\[
  \jmath_0\psi:=
  \begin{bmatrix}\psi\\0\end{bmatrix},
  \qquad
  \jmath_1\psi:=
  \begin{bmatrix}0\\\psi\end{bmatrix}.
\]

Since \(h(0)=h(1)\), the rule
\[
  \rho_h\bigl(\emb([t])\bigr):=h(t)
\]
defines a CPwL function on the polygonal loop \(\Gamma\).  Fix a global
CPwL extension, still denoted by \(\rho_h\).

For \(x\in[0,1]\), let \(Y_0(x):=\emb([x/2])\),
\(Y_j(x):=F^j(Y_0(x))\), and
\(x_j:=\Pi(Y_j(x))=\tau^j(x)\).  Since
\(Y_{n+1}(x)=\emb([2^n x])\), define
\begin{equation}
\label{eq:adjoint-terminal-scalar}
  a_n(x):=\rho_h(Y_{n+1}(x))=h(R^n(x)),
\end{equation}
with endpoint values understood by continuity.

For \(\lambda\in E^*\), initialize
\begin{equation}
\label{eq:adjoint-initial-state}
  \Lambda_0(x):=a_n(x)\jmath_0\lambda,
\end{equation}
and, for \(j=0,\ldots,n-1\), set
\begin{equation}
\label{eq:adjoint-folded-update}
  \Lambda_{j+1}(x)
  :=
  \widehat T^{\trans}
  \mathcal C_{K_{j,n}}\bigl(x_j,\Lambda_j(x)\bigr),
\end{equation}
where the cone scales \(K_{j,n}\) are specified below.  Thus the second
pass reruns \(Y_j\) forward while propagating the adjoint state.

Set \(T_*^\vee:=\|\widehat T^{\trans}\|\),
\(L_h:=\Lip(h)\), and \(P_-:=(I-\mathcal J)/2\).  For fixed \(n\), choose
\begin{equation}
\label{eq:adjoint-cone-scales}
  K_{j,n}
  :=
  1+
  2^{\,n-j}\|P_-\|(T_*^\vee)^j
  L_h\|\lambda\|_\infty,
  \qquad j=0,\ldots,n-1.
\end{equation}

\begin{lemma}[Adjoint fold-cone estimate]
\label{lem:adjoint-fold-cone}
If the updates through stage \(j-1\) are exact, then
\begin{equation}
\label{eq:adjoint-cone-bound}
  \|P_-\Lambda_j(x)\|_\infty
  \le
  K_{j,n}|x_j-\textstyle\frac12|.
\end{equation}
Hence the cone switch at stage \(j\) acts as the identity for
\(x_j\le\frac12\) and as \(\mathcal J\) for \(x_j\ge\frac12\).
\end{lemma}

\begin{proof}
Exactness of the preceding updates gives
\[
  \|\Lambda_j(x)\|_\infty
  \le
  (T_*^\vee)^j\|\lambda\|_\infty |h(R^n(x))|.
\]
Since
\(R^n(x) \in \{\tau^{n-j}(x_j),1-\tau^{n-j}(x_j)\}\)
and \(\tau^{n-j}(\frac12)\in\{0,1\}\), the endpoint conditions on \(h\)
imply
\[
  |h(R^n(x))|
  \le
  L_h\,2^{\,n-j}
  |x_j-\textstyle\frac12|.
\]
Applying \(P_-\) yields \eqref{eq:adjoint-cone-bound}, and
\Cref{prop:involution-cone-switch} gives the conclusion.  At
\(x_j=\frac12\), the terminal factor vanishes, so \(\Lambda_j(x)=0\).
\end{proof}

\subsection{Exact realization and globalization}

We now identify the adjoint output with the cascade and recover the vectorized atom.

\begin{proposition}[Adjoint realization of a unit-interval atom]
\label{prop:adjoint-unit-atom}
For every \(\lambda\in E^*\), the recursion
\eqref{eq:adjoint-initial-state}--\eqref{eq:adjoint-folded-update}
satisfies
\begin{equation}
\label{eq:adjoint-final-pairing}
  \bigl\langle\Lambda_n(x),\widehat u\bigr\rangle
  =
  \bigl\langle\lambda,G_n(x)\bigr\rangle,
  \qquad x\in[0,1].
\end{equation}
This pairing admits an exact ReLU realization of fixed width and depth
\(O(n)\), with weights and biases bounded by \(C\Lambda^n\), for constants
independent of \(n\).
\end{proposition}

\begin{proof}
By \Cref{lem:adjoint-fold-cone}, every cone switch performs the ideal
branch operation.  We first work away from the dyadic points of level
\(n\), and then extend the result by continuity.

Set \(q_j:=Q(R^j(x))\), \(j=0,\ldots,n-1\).  Choose
\(\varepsilon_j\in\{0,1\}\), with \(\varepsilon_0=0\), so that
\(R^j(x)=x_j\) for \(\varepsilon_j=0\) and
\(R^j(x)=1-x_j\) for \(\varepsilon_j=1\).  If
\(\delta_j=0\) for \(x_j<\frac12\) and \(\delta_j=1\) for
\(x_j>\frac12\), then
\[
  q_j=\varepsilon_j\mathbin{\oplus}\delta_j,
  \qquad
  \varepsilon_{j+1}=q_j.
\]

Suppose that
\(\Lambda_j=a_n\,\jmath_{\varepsilon_j}\psi_j\).  The cone switch moves
the active sheet to \(q_j\), and the subsequent adjoint update gives
\[
  \Lambda_{j+1}
  =
  a_n\,\jmath_{\varepsilon_{j+1}}
  \bigl(T_{q_j}^{\trans}\psi_j\bigr).
\]
Since \(\psi_0=\lambda\), induction yields
\[
  \psi_n
  =
  T_{q_{n-1}}^{\trans}\cdots T_{q_0}^{\trans}\lambda.
\]
Using \(a_n=h(R^n(x))\) and \(\widehat u=(u,u)\), we obtain
\[
  \langle\Lambda_n,\widehat u\rangle
  =
  h(R^n(x))
  \left\langle
    T_{q_{n-1}}^{\trans}\cdots T_{q_0}^{\trans}\lambda,u
  \right\rangle 
  =
  \left\langle
    \lambda,
    T_{q_0}\cdots T_{q_{n-1}}h(R^n(x))u
  \right\rangle
  =
  \langle\lambda,G_n(x)\rangle.
\]

A first loop pass of length \(n+1\) computes \(a_n\) while carrying the
input \(x\).  A second pass reruns \(Y_j\) forward and applies one adjoint
update at each stage, followed by the linear pairing with \(\widehat u\).
All repeated modules have fixed complexity, so the width is fixed and the
depth is \(O(n)\).  Finally,
\eqref{eq:adjoint-cone-scales} gives
\(K_{j,n}\le C\max\{2,T_*^\vee,1\}^n\), which yields the stated
exponential coefficient bound.
\end{proof}

\begin{corollary}[Unit-interval homogeneous atom]
\label{cor:adjoint-unit-atom}
The vectorization \(\vect(V^n\gamma_{h,v})\) on \([0,1]\) admits an exact
ReLU realization of fixed width and depth \(O(n)\), with weights and
biases bounded by \(C\Lambda^n\).
\end{corollary}

\begin{proof}
Apply \Cref{prop:adjoint-unit-atom} in parallel to the standard coordinate
covectors of \(E^*\).
\end{proof}

\begin{remark}[Passage to general homogeneous data]
\label{rem:adjoint-globalization}
A sufficiently fine nodal-hat decomposition writes every compactly
supported vector-valued CPwL seed as a finite sum of translates of
endpoint-zero unit-interval atoms.  Translation covariance, finite
parallelization, and standard clamped gluing then give the global
homogeneous realization.  These reductions originate in \cite{source}
and were subsequently used for vector-valued refinement in \cite{loop};
we omit the familiar details.
\end{remark}

\section{Polyphase shifts and fan switching for fixed arity}
\label{app:fan-switch}

This appendix extends the folded construction to every fixed arity
\(M\ge2\).  Although the \(M\) branch maps need not form a symmetry orbit,
their digit dependence factors through finite coordinate shifts.  After
reflection doubling, adjacent shift--reflection candidates agree at the
folds of the \(M\)-branch zigzag map, so a finite chain of cone switches,
called a fan switch, yields an exact folded recursion.  All constants and
network widths below may depend on \(M\).

\subsection{Polyphase factorization and zigzag folding}

Let
\[
  (V_Mf)(t)=\sum_{j\in\Z}A_jf(Mt-j),
\]
where the matrix mask is finitely supported, and assume that \(V_M\)
preserves \([0,L]\).  Set \(E=(\R^p)^L\), with blocks indexed by
\(k=0,\ldots,L-1\), and define
\[
  (\vect f)(x)
  =
  \bigl(f(x),f(x+1),\ldots,f(x+L-1)\bigr)^{\trans},
  \qquad x\in[0,1].
\]
For \(q=0,\ldots,M-1\), let \(T_q:E\to E\) be given by
\begin{equation}
\label{eq:mary-transition}
  (T_qu)_k
  =
  \sum_{\ell=0}^{L-1}A_{q+Mk-\ell}u_\ell .
\end{equation}

To expose the shift structure, introduce the padded fiber
\[
  I^+:=\{-(M-1),\ldots,L-1\},
  \qquad
  E^+:=(\R^p)^{I^+}.
\]
Extend \(u\in E\) by zero outside \(\{0,\ldots,L-1\}\), and define
\[
  (\Sigma_qu)_r:=u_{r+q},
  \qquad r\in I^+.
\]
Finally, define the common polyphase operator
\(\mathsf T:E^+\to E\) by
\begin{equation}
\label{eq:common-polyphase}
  (\mathsf Tw)_k
  :=
  \sum_{r\in I^+}A_{Mk-r}w_r .
\end{equation}

\begin{lemma}[Padded polyphase factorization]
\label{lem:polyphase-factorization}
For \(q=0,\ldots,M-1\), we have
\begin{equation}
\label{eq:Tq-polyphase}
  T_q=\mathsf T\Sigma_q.
\end{equation}
Moreover, \(\|\Sigma_q\|\le1\) in the max norm.
\end{lemma}

\begin{proof}
With \(r=\ell-q\), we have
\[
  (\mathsf T\Sigma_qu)_k
  =
  \sum_{\ell=0}^{L-1}A_{q+Mk-\ell}u_\ell
  =
  (T_qu)_k.
\]
The norm bound is immediate.
\end{proof}

For \(x\in[q/M,(q+1)/M]\), define
\begin{equation}
\label{eq:mary-zigzag}
  \tau_M(x)
  :=
  \begin{cases}
    Mx-q,&q\ \text{even},\\
    q+1-Mx,&q\ \text{odd}.
  \end{cases}
\end{equation}
Then \(\tau_M:[0,1]\to[0,1]\) is CPwL, and for
\(R_M(x):=Mx-\lfloor Mx\rfloor\) one has
\[
  R_M(x)=
  \begin{cases}
    \tau_M(x),&q\ \text{even},\\
    1-\tau_M(x),&q\ \text{odd},
  \end{cases}
\]
on the \(q\)-th branch.

Set \(\widehat E:=E\oplus E\) and
\(\widehat E^+:=E^+\oplus E^+\), and let \(\mathcal J\) denote the swap
involution on either space.  Writing \(\bar q:=M-1-q\), define
\begin{equation}
\label{eq:Pq-definition}
  \mathcal P_q
  :=
  \begin{bmatrix}
    \Sigma_q&0\\
    0&\Sigma_{\bar q}
  \end{bmatrix}
  \mathcal J^q
  :
  \widehat E\to\widehat E^+,
\end{equation}
where \(\mathcal J^q=I\) for even \(q\) and
\(\mathcal J^q=\mathcal J\) for odd \(q\).  Also set
\[
  \widehat{\mathsf T}
  :=
  \begin{bmatrix}
    \mathsf T&0\\
    0&\mathsf T
  \end{bmatrix}.
\]

Let \(W_Mf:=V_Mf+B\), \(b:=\vect B\), and
\(G_m:=\vect(W_M^mf_0)\).  Define
\[
  G_m^\#(x):=
  \begin{bmatrix}G_m(x)\\G_m(1-x)\end{bmatrix},
  \qquad
  b^\#(x):=
  \begin{bmatrix}b(x)\\b(1-x)\end{bmatrix}.
\]

\begin{proposition}[Fixed-arity folded recursion]
\label{prop:mary-folded-branches}
For \(x\in[q/M,(q+1)/M]\),
\begin{equation}
\label{eq:mary-folded-branch}
  G_{m+1}^\#(x)
  =
  \widehat{\mathsf T}\,
  \mathcal P_qG_m^\#(\tau_M(x))
  +
  b^\#(x).
\end{equation}
\end{proposition}

\begin{proof}
On the \(q\)-th branch,
\(G_{m+1}(x)=T_qG_m(R_M(x))+b(x)\).  Since \(1-x\) lies in branch
\(\bar q\) and \(R_M(1-x)=1-R_M(x)\), reflection doubling gives the
branch matrix
\[
  \begin{bmatrix}
    T_q&0\\
    0&T_{\bar q}
  \end{bmatrix}
  \mathcal J^q,
\]
where \(\mathcal J^q=I\) for even \(q\) and
\(\mathcal J^q=\mathcal J\) for odd \(q\).  By
\Cref{lem:polyphase-factorization}, this matrix equals
\(\widehat{\mathsf T}\mathcal P_q\), which proves
\eqref{eq:mary-folded-branch}.
\end{proof}

The key point is that adjacent branch maps agree at the folds before the
common operator \(\widehat{\mathsf T}\) is applied.
Let \(\mathsf D:E\to E\) be the cell shift
\((\mathsf Du)_k:=u_{k+1}\), with zero extension at the last block.  If
\(G=\vect f\) for a continuous function supported in \([0,L]\), then we have
\begin{equation}
\label{eq:G-endpoint-shift}
  G(1)=\mathsf DG(0).
\end{equation}
Moreover, if \(u_0=0\), then one has
\begin{equation}
\label{eq:Sigma-D}
  \Sigma_a\mathsf Du=\Sigma_{a+1}u,
  \qquad a=0,\ldots,M-2,
\end{equation}
with the only possible boundary discrepancy being the vanishing entry \(u_0\).

\begin{lemma}[Adjacent fold agreement]
\label{lem:fan-fold-agreement}
Let \(G=\vect f\) for a continuous function supported in \([0,L]\), and
set \(c_q:=q/M\).  Then, for \(q=1,\ldots,M-1\),
\begin{equation}
\label{eq:pre-mask-agreement}
  \mathcal P_{q-1}G^\#(\tau_M(c_q))
  =
  \mathcal P_qG^\#(\tau_M(c_q)).
\end{equation}
\end{lemma}

\begin{proof}
Write \(g:=G(0)\).  Then \(g_0=f(0)=0\) and
\(G(1)=\mathsf Dg\).  Since \(\tau_M(c_q)=1\) for odd \(q\) and
\(\tau_M(c_q)=0\) for even \(q\), \eqref{eq:Sigma-D} shows that both
sides of \eqref{eq:pre-mask-agreement} equal
\((\Sigma_qg,\Sigma_{M-q}g)^\trans\).
\end{proof}

\subsection{The fan switch}

For vectors \(u,v\) in a fixed coordinate space and \(c\in\R\), define
\begin{equation}
\label{eq:agreement-switch}
  \mathcal A_{K,c}(x;u,v)
  :=
  \frac{u+v}{2}
  +
  \Phi_K\big(x-c,\frac{u-v}{2}\big),
\end{equation}
where \(\Phi_K\) is the coordinatewise cone flip from
\eqref{eq:scalar-cone-flip}.  By \Cref{lem:scalar-cone-flip}, whenever
\(\frac12\|u-v\|_\infty\le K|x-c|\),
\begin{equation}
\label{eq:agreement-switch-exact}
  \mathcal A_{K,c}(x;u,v)
  =
  \begin{cases}
    u,&x\le c,\\
    v,&x\ge c.
  \end{cases}
\end{equation}
Thus \(\mathcal A_{K,c}\) is a global CPwL agreement switch of fixed
width and depth, with weights bounded by \(C(1+K)\).
For \(u\in\widehat E\), define
\begin{align}
  \mathfrak F_{K,0}(x,u)
  &:=
  \mathcal P_0u,
  \label{eq:fan-initial}\\
  \mathfrak F_{K,q}(x,u)
  &:=
  \mathcal A_{K,c_q}
  \bigl(x;\mathfrak F_{K,q-1}(x,u),\mathcal P_qu\bigr),
  \qquad q=1,\ldots,M-1,
  \label{eq:fan-recursion}
\end{align}
and set \(\mathfrak F_K:=\mathfrak F_{K,M-1}\).  Since \(M\) is fixed,
this \(M\)-branch fan switch is a fixed-width, fixed-depth CPwL module.

\begin{proposition}[Exact fan switching]
\label{prop:exact-fan-switch}
Let \(G=\vect f\) be Lipschitz, where \(f\) is continuous and supported in
\([0,L]\).  If
\(K\ge M\Lip(G^\#)\),
then, for \(x\in[q/M,(q+1)/M]\), we have
\begin{equation}
\label{eq:fan-exactness}
  \mathfrak F_K\bigl(x,G^\#(\tau_M(x))\bigr)
  =
  \mathcal P_qG^\#(\tau_M(x)).
\end{equation}
\end{proposition}

\begin{proof}
Set \(H_q(x):=\mathcal P_qG^\#(\tau_M(x))\).  Since
\(\|\mathcal P_q\|\le1\) and \(\Lip(\tau_M)=M\), each \(H_q\) is Lipschitz with its constant bounded by \(M\Lip(G^\#)\).  
By \Cref{lem:fan-fold-agreement}, one has \(H_{q-1}(c_q)=H_q(c_q)\).

After the first \(q-1\) switches, the valid output is the continuous
piecewise curve assembled from \(H_0,\ldots,H_{q-1}\).  It has Lipschitz
constant at most \(M\Lip(G^\#)\) and agrees with \(H_q\) at \(c_q\).
Hence we have
\[
  \textstyle\frac12\|H_{\mathrm{old}}(x)-H_q(x)\|_\infty
  \le
  M\Lip(G^\#)|x-c_q|
  \le K|x-c_q|.
\]
Thus \eqref{eq:agreement-switch-exact} makes the \(q\)-th switch exact.
Induction over \(q\) proves \eqref{eq:fan-exactness}.
\end{proof}

The preceding propositions give the folded update used in the realization:
\begin{equation}
\label{eq:mary-fan-recursion}
  G_{m+1}^\#(x)
  =
  \widehat{\mathsf T}\,
  \mathfrak F_{K_m}\bigl(x,G_m^\#(\tau_M(x))\bigr)
  +
  b^\#(x),
\end{equation}
where \(K_m:=1+M\Lip(G_m^\#)\).  If
\(L_m:=\Lip(G_m^\#)\) and \(L_b:=\Lip(b^\#)\), then
\(L_{m+1}\le M\|\widehat{\mathsf T}\|L_m+L_b\).  Hence, for constants
\(C_L,C_K>0\) and \(\Lambda\ge1\) depending only on the fixed data,
\(L_m\le C_L\Lambda^m\) and \(K_m\le C_K\Lambda^m\).

\subsection{Realization consequence}

With the reflection quotient
\(\varpi:\T\to[0,1]\) defined by \(\varpi([t]):=2\operatorname{dist}(t,\Z)\), we have
\[
  \varpi([Mt])=\tau_M(\varpi([t])),
\]
so the residual memory construction of
\cite[\S3.2]{affine-memory} applies with
\(F_M(\emb([t])):=\emb([Mt])\).  For fixed \(M\), the nontrivial
preimages are uniformly separated on the loop, and the resulting injective
CPwL skew product gives exact backward replay of the orbit
\(\tau_M^j(x)\).

\begin{theorem}[Fixed-arity folded realization]
\label{thm:mary-fan-realization}
Fix \(M\ge2\), and let
\[
  (W_Mf)(t)
  =
  \sum_{j\in\Z}A_jf(Mt-j)+B(t)
\]
have finitely supported matrix mask.  Assume that its homogeneous part
preserves \([0,L]\), and that \(f_0,B:\R\to\R^p\) are CPwL and supported
there.  Then we have
\[
  W_M^nf_0\in\Ups_{C_0,C_1n}(\ReLU;1,p),
  \qquad n\ge1,
\]
for constants \(C_0,C_1>0\) independent of \(n\).  The weights and biases
may be bounded by \(C_2\Lambda^n\), with constants depending only on the
fixed refinement and CPwL data.
\end{theorem}

\begin{proof}
Let \(x_j:=\tau_M^j(x)\).  The memory controller supplies
\(x_n,\ldots,x_0\) in reverse order.  Starting from
\(U_n:=G_0^\#(x_n)\), define
\[
  U_j
  :=
  \widehat{\mathsf T}\,
  \mathfrak F_{K_{n-1-j}}(x_j,U_{j+1})
  +
  b^\#(x_j),
  \qquad j=n-1,\ldots,0.
\]
By \eqref{eq:mary-fan-recursion},
\(U_j=G_{n-j}^\#(x_j)\) by backward induction.  Since \(M\) is fixed,
all modules have fixed complexity, giving fixed width and depth \(O(n)\).
The coefficient bound follows from the Lipschitz estimate above,
and standard clamped gluing recovers the
global iterate.
\end{proof}

\begin{remark}[Homogeneous adjoint form]
\label{rem:mary-adjoint-fan}
The factorization also extends the memory-free homogeneous adjoint
construction.  Since
\[
  \bigl(\widehat{\mathsf T}\mathcal P_q\bigr)^{\trans}
  =
  \mathcal P_q^{\trans}\widehat{\mathsf T}^{\trans},
\]
each stage applies the common adjoint mask operator and then fan-switches
among the shift--reflection candidates
\(\mathcal P_q^{\trans}\widehat{\mathsf T}^{\trans}\Lambda\).
For endpoint-zero atoms, the terminal factor vanishes at every zigzag fold,
and the same Lipschitz estimate gives the adjacent cone bounds.  Thus the
binary adjoint construction extends to every fixed \(M\), with its single
cone switch replaced by the \(M\)-branch fan switch.
\end{remark}

\section*{Acknowledgments}

This work was supported by the Natural Sciences and Engineering Research Council of Canada
through its Discovery Grants program.

\end{document}